\documentclass[a4paper,12pt]{article}
\usepackage{graphicx}
\usepackage{latexsym}
\usepackage{amsmath}
\usepackage{amssymb}
\usepackage{enumerate}
\usepackage{pict2e}
\usepackage{xcolor}
\usepackage{tikz}
\usepackage{float}
\usepackage{cancel}
\usepackage{theorem}
\newtheorem{theorem}{Theorem}[section]
\newtheorem{Proposition}[theorem]{Proposition}
\newtheorem{Lemma}[theorem]{Lemma}

\theorembodyfont{\rmfamily}
\newtheorem{proof}{\textmd{\textit{Proof.}}}

\newtheorem{Remark}[theorem]{Remark}
\newtheorem{Example}[theorem]{Example}
\newtheorem{Definition}[theorem]{Definition}

\makeatletter

\@addtoreset{equation}{section}
\makeatother

\newcommand{\qedd}{\hfill \Box}

\newcommand{\R}{\ensuremath{\mathbb{R}}}
\newcommand{\Sph}{\ensuremath{\mathbb{S}}}

\newcommand{\cK}{\ensuremath{\mathcal{K}}}
\newcommand{\cL}{\ensuremath{\mathcal{L}}}

\newcommand{\cP}{\ensuremath{\mathcal{P}}}

\title{Randers metrics on two-spheres of revolution with simple cut locus \footnote{Mathematics Subject Classification: 53C22, 53C60}
\footnote{
Keywords: Riemannian manifolds; geodesics; conjugate locus; two-sphere of revolution; Randers
metric; Finslers metric; cut locus
}}
\author{Rattanasak Hama and Sorin V. Sabau}

\date{\today}
\begin{document}

\maketitle

\begin{abstract}
In the present paper, we study the Randers metric on two-spheres of revolution in order to obtain new families of Finsler of Randers type metrics with simple cut locus. We determine the geodesics behavior, conjugate and cut loci of some families of Finsler metrics of Randers type whose navigation data is not a Killing field and without sectional or flag curvature restrictions. Several examples of Randers metrics whose cut locus is simple are shown.
\end{abstract}

\section{Introduction}

A two-sphere of revolution is a compact Riemannian surface $(M,h)$, which is homeomorphic to the sphere $\Sph^2\subset \R^3$. If this manifold is endowed with a Randers metric $F=\alpha+\beta$, or more generally, with an arbitrary positive defined Finsler metric $F$, then $(M,F)$ is called a Randers or Finsler two-sphere of revolution, respectively.

One of the major problems in Differential Geometry (see \cite{SST}, \cite{ST}) and Optimal Control (see \cite{BCST}) is the study of geodesics, conjugate points and cut points of Riemannian or Finsler manifolds. We recall that a vector field $J$ along a unit speed geodesic $\gamma:[0,a]\to M$ is said to be a Jacobi field if it satisfies the well-known Jacobi equation (see for instance \cite{BCS}, Chapter 7 for details). A point $p$ is said to be conjugate to $q:=\gamma(0)$ along $\gamma$ if there exists a non-zero Jacobi field $J$ along $\gamma$ which vanishes at $p$ and $q$. The set of conjugate points of $q$ along all curves $\gamma$ starting at $q$ is called the conjugate locus of $q$.

If $\gamma:[0,l]\to M$ is a minimal geodesic on a such manifold, then its end point $\gamma(l)\in M$ is called the cut point of the initial point $q=\gamma(0)\in M$, in the sense that any extension of $\gamma$ beyond $\gamma(l)$ is not a minimizing geodesic from $q$ anymore. The cut locus $Cut(q)$ is defined as the set of cut points of $q$, and on Riemannian or Finslerian surfaces, it has  the structure of a local tree. Moreover, the cut points $p\in Cut(q)$ of $q$ are characterized by the property that the distance $d(q,\cdot)$ from $q$ is not smooth any more at $p$ (see \cite{ST2016} for details). The cut points $p$ along a geodesic $\gamma$ emanating from the point $q=\gamma(0)$ can appear either before or at the first conjugate point of $q$ along $\gamma$, but not after that (see \cite{BCS}). 

To determine the precise structure of the cut locus on a Riemannian or Finsler manifold is not an easy task. The majority of known results concern Riemannian or Randers surfaces of revolution (see \cite{ST}, \cite{TASY} for the Riemannian, and \cite{HKS}, \cite{HS} for the Randers case).   

A Randers metric $F=\alpha+\beta$ is a special Finsler metric obtained by the deformation of a Riemannian metric $\alpha$ by a one-form $\beta$ whose Riemannian $\alpha$-length is less than one in order to assure that $F$ is positively defined (\cite{Ra}). 
These Finsler metrics are intuitive generalizations of the Riemannian ones having most of the geometrical objects relatively easy to compute (see \cite{BCS}). 

An equivalent characterization of Randers metrics is through the Zermelo's navigation problem. We recall that a Finsler metric $F$ is characterized by its indicatrix $\{(x,y)\in TM: F(x,y)=1\}$ (see~\cite{BCS}). In particular, a Randers metric indicatrix is obtained as the rigid translation of the unit sphere $\{y\in T_xM: h(x,y)=1\}$ of a Riemannian metric $(M,h)$ by a vector field $W\in \mathcal{X}(M)$ whose Riemannian length is less than one. The pair $(h,W)$ will be called the navigation data of the Randers metric $F=\alpha+\beta$. Conversely, the Randers metric $F=\alpha+\beta$ will be called the solution of Zermelo's navigation problem $(h,W)$. In the case when $W$ is an $h$-Killing field, provided $h$ is not flat, the geodesics, conjugate points and cut points  of the Randers metric $F=\alpha+\beta$ can be obtained by the translation of the corresponding geodesics, conjugate points and cut points,  of the Riemannian metric $h$ by the flow of $W$, respectively (see \cite{HS}, \cite{R}). More generally, new Finsler metrics $F$ can be obtained by the rigid translation of the indicatrix of a given Finsler metric $F_0$ by a vector field $W$, such that $F_0(-W)<1$ (see \cite{FM}, \cite{S}). In this case, the pair $(F_0,W)$ will be called the general navigation data of $F$. 

Another case when the Randers geodesics can be easily related to the Riemannian ones is when the deformation one-form $\beta$ is closed. Indeed, the Randers metric $F=\alpha+\beta$ is projectively equivalent to the underlying Riemannian metric $\alpha$ if and only if $d\beta=0$. In this case, the $\alpha$-geodesics, conjugate points and cut points  coincide with the $F$-geodesics, cut points and conjugate points, respectively (see \cite{BCS}).

We combine these two cases of Randers metrics in order to obtain new families of Finsler of Randers type with simple cut locus (see Section \ref{sec_two_sphere} for the definition). The originality of our paper lies in the followings:
\begin{enumerate}[(i)]
\item We determine the geodesics behavior, conjugate and cut loci of some families of Finsler metrics of Randers type whose navigation data do not necessarily include a Killing field. 
\item We show that the structure of the cut locus of these families can be determined without any sectional or flag curvature restrictions. These are generalizations of the results in \cite{TASY} to the Randers case.
\item We construct a sequence of Randers metrics whose cut locus structure is simple.
\item We extend some classical results from the case of Randers metrics to $\beta$-changes of Finsler metrics and give new proofs to some known results. 
\end{enumerate}

If we start with a Riemannian two-sphere of revolution $(M\simeq \Sph^2,h)$ and the vector fields $V_0,V,W\in \mathcal{X}(M)$. Then the following construction gives the Randers metric $F_0,~F_1$ and $F_2$ as solutions of the Zermelo's navigation problem with data $(h,V_0), (F_0,V)$ and $(F_1,W)$, respectively,
\begin{figure}[H]
\begin{center}
\setlength{\unitlength}{1cm}
\begin{picture}(20,1)

\put(1.2,0.1){\vector(1,0){1.8}}
\put(6.3,0.1){\vector(1,0){1.8}}
\put(11.4,0.1){\vector(1,0){1.8}}

\put(0,0){$(M,h)$}
\put(3,0){$(M,F_0=\alpha_0+\beta_0)$}
\put(8.1,0){$(M,F_1=\alpha_1+\beta_1)$}
\put(13.2,0){$(M,F_2=\alpha_2+\beta_2),$}

\put(1.5,0.4){$\quad V_0$}
\put(6.6,0.4){$\quad V$}
\put(11.7,0.4){$\quad W$}

\end{picture}
\end{center}
\end{figure}

\noindent
which are positively defined provided $\|V_0\|_h<1$, $F_0(-V)<1$, and $F_1(-W)<1$, respectively. If we impose conditions that $V_0$ and $V$ to be $h$- and $F_0$-Killing fields, respectively, and $d\beta_2=0$, then the geodesics, conjugate and cut loci of $F_2$ can be determined. 

Remarkably, a shortcut of this construction would be to simply impose $\|V_0+V+W\|_h<1$, that guarantees $F_2$ is positively defined, and $V_0+V+W$ to be $h$-Killing. In this case, $F_2$ is also with simple cut locus having the same structure with the cut locus of $h$. Obviously, the cut loci of these metrics are slightly different as set of points on $M$.

This construction can be extended to a sequence of Randers metrics $\{F_i=\alpha_i+\beta_i\}_{i=1\ldots,n}$ whose cut loci are simple (see Remark \ref{rem: sequence}). 

Here is the structure of our paper. 

In Section \ref{sec_two_sphere}, we review the geometry of Riemannian two-spheres of revolution from \cite{ST} and \cite{TASY}.

In Section \ref{sec: three Randers sphere}, we describe the geometry of some families of Randers metrics obtained as generalizations to the Finslerian case of the Riemannian metrics in \cite{TASY}. We use the Hamiltonian formalism for giving and proving some basic results to be used later in the section. Lemma \ref{lem_A1} is an important result that generalizes a well-known result (\cite{MHSS}) for Randers metrics to more general Finsler metrics obtained by $\beta$-changes. The relation with $F$-Killing fields are given in Lemma \ref{lem_A2} and the basic properties of  our family of Randers metrics are in Lemma \ref{lem_new_1}.  Some of these results are indirectly suggested in \cite{FM}, but here we clarify all the basic aspects and prove them in our specific formalism. Lemma \ref{lem_new_2} gives the concrete expressions of $\widetilde{\alpha}$ and $\widetilde{\beta}$ in the families of our Randers metric, formulas that provide a better understanding of the positive definiteness of these metrics. 
 
Lemma \ref{lem: X} gives the behavior of geodesics, conjugate and cut points of the $\beta$-change of a Randers metric generalizing the results in \cite{R}. Lemma \ref{lem: X2} gives the conditions for the one-form $\beta$ to be closed in terms of the navigation data. Finally, we sum up the results in all these lemmas in Theorem \ref{thm_main_1}, which is the main result of the present paper. In Remark \ref{rem: sequence}, we show how an infinite sequence of such Randers metrics can be obtained. 
 
In Section \ref{sec: four Examples}, we construct one example of the Randers metric on the two-sphere of revolution that satisfies the conditions in Theorem \ref{thm_main_1}.

\section{Two-spheres of revolution}\label{sec_two_sphere}

Classically, surfaces of revolution are obtained by rotating a curve $(c)$ in the $xz$ plane around the $z$ axis. More precisely, if the profile curve $(c)$ is given parametrically
\begin{equation}\label{eq_parametric}
(c):\begin{cases}
x=\varphi(u)\\
z=\psi(u)
\end{cases},\
\varphi>0,\
u\in I\subset \R,
\end{equation}
then, in the case $(\varphi'(u))^2+(\psi'(u))^2\neq 0$, for all $u\in I$, the curve can be written explicitly $x=f(z)$ or implicitly by $\Phi(x,z)=0$, where $'$ is the derivative with respect to $u$.

In the case of the parametric representation \eqref{eq_parametric} one obtains an $\R^3$-immersed surface of revolution $\psi:\Omega\subset \R^2\to\R^3$, given by
\begin{equation}\label{eq_immersed}
\psi(u,v)=\left(\varphi(u)\cos v,\varphi(u)\sin v,\psi(u)\right),\ u\in I,\ v\in [0,2\pi).
\end{equation}

\begin{Remark}
The immersed surface of revolution \eqref{eq_immersed} is called of {\it elliptic type}, while
$$
\psi(u,v)=\left(\varphi(u)\cosh v,\varphi(u)\sinh v,\psi(u)\right)
$$
is called a {\it hyperbolic type}. Since we are interested in compact surfaces of revolution, only the elliptic case will be considered hereafter, leaving the hyperbolic type for a future research.
\end{Remark}
Even though the representation \eqref{eq_immersed} is quite intuitive, it has two major disadvantages:
\begin{enumerate}
\item[(1)] it leads to quite complicated formulas for the induced Riemannian metric, geodesic equations, Gauss curvature, etc.,
\item[(2)] it excludes the case of abstract  surfaces  of revolution which cannot be embedded in $\R^3$.
\end{enumerate}

The first disadvantage can be easily fixed by taking the curve $(c)$ to be unit speed parameterized in the Euclidean plane $xz$, i.e.,
$$
[\varphi'(u)]^2+[\psi'(u)]^2=1,
$$
which leads to the warped Riemannian metric
$$
ds^2=du^2+\varphi^2(u)dv^2.
$$

This simplification suggests the following definition (which also fixes the second disadvantage).
\begin{Definition}(\cite{ST})
Let $(M,h)$ be a compact Riemannian surface homeomorphic to $\Sph^2$. If $M$ admits a {\it pole} $p\in M$, and for any two points $q_1,q_2\in M$, such that $d_h(p,q_1)=d_h(p,q_2)$, there exists an Riemannian isometry $i:M\to M$ for which
$$
i(q_1)=q_2,\ i(p)=p,
$$
then $(M,h)$ is called a {\it two-sphere of revolution}. Here $d_h$ is the distance function associated to the Riemannian metric $h$.
\end{Definition}

\begin{Remark}
One example of compact surface of revolution that cannot be embedded in $\R^3$ is the real projective space $\R P^2$. It is compact being homeomorphic to $\Sph^2/_\sim$, where $\Sph^2$ is the unit sphere in $\R^3$ and $\sim$ is the equivalence relation $x\sim -x$, for all $x\in\Sph^2$. It is a surface of revolution because it can be obtained by rotating the M\"obius strip along its center line.

Finally, it cannot be embedded in $\R^3$ because it is non-orientable. More generally, it is known that any embedding of a non-orientable surface in $\R^3$ must create self-intersections and this is not allowed. Nevertheless, $\R P^2$ can be immersed in $\R^3$, and therefore can be locally embedded in $\R^3$, but not globally (see \cite{B} for properties of projective spaces).

{\spaceskip=0.15em\relax Another example is the so-called Lorentz surface, obtained by rotating the hyperbola $x^2-y^2=1$ around the $x$-axis. This surface is orientable but cannot embedded in $\R^3$ because it has a self-intersection at origin.}
\end{Remark}

This definition allows to introduce the geodesic polar coordinates $(r,\theta)\in(0,2a)\times[0,2\pi)$ around $p$, such that the Riemannian metric is given as
$$
h=dr^2+m^2(r)d\theta^2
$$
on $M\setminus \{p,q\}$, where $q$ is the unique cut point of $p$ and
$$
m(r)=\sqrt{h\left(\frac{\partial}{\partial \theta},\frac{\partial}{\partial\theta}\right)}
$$
(see \cite{SST} or \cite{ST} for details).

Moreover, the functions $m(r)$ and $m(2a-r)$ can be extended to smooth odd function around $r=0$, where $d_h(p,q)=2a$, $m'(0)=1=-m'(2a)$.

It is well-known that any pole $p\in M$ must have a unique cut point $q\in M$, and that any geodesic starting from $p$ contains $q$. 

For the sake of simplicity, we consider $a=\frac{\pi}{2}$, that is $m:[0,\pi]\to[0,\infty)$ will satisfy $m(0)=0$, $m'(0)=1$, $m(\pi-r)=m(r)>0$, for all $r\in(0,\pi)$, see Figure \ref{fig_two-sphere}.

\begin{figure}[H]
	\begin{center}
\includegraphics[scale=0.3]{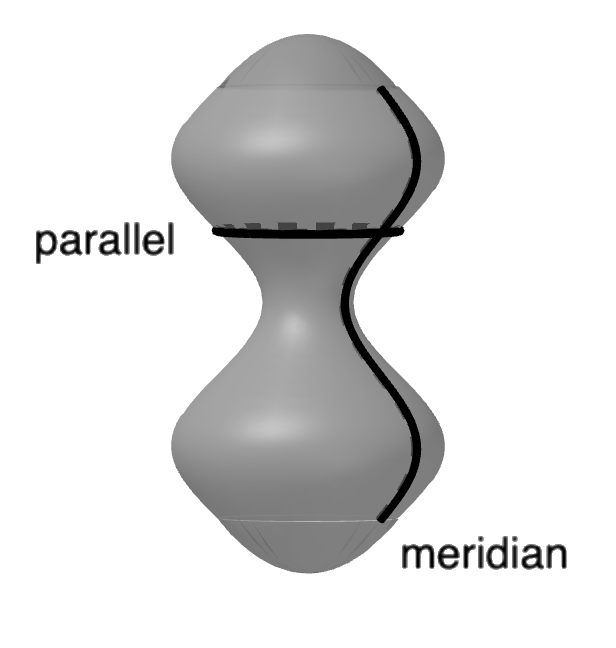}
	\end{center}
	\caption{A two-sphere of revolution.}
	\label{fig_two-sphere}
\end{figure}

Recall (see \cite{SST}) that the equations of an $h$-unit speed geodesic $\gamma(s):=(r(s),\theta(s))$ of $(M,h)$ are 
\begin{equation*}
\begin{cases}
\dfrac{d^2r}{ds^2}-mm'\left(\dfrac{d\theta}{ds}\right)^2=0, \vspace{0.5cm}\\
\dfrac{d^2\theta}{ds^2}+2\dfrac{m'}{m}\left(\dfrac{dr}{ds}\right)\left(\dfrac{d\theta}{ds}\right)=0,
\end{cases}
\end{equation*}
where $s$ is the arclength parameter of $\gamma$ with the $h$-unit speed parameterization condition
$$
\left(\dfrac{dr}{ds}\right)^2+m^2\left(\dfrac{d\theta}{ds}\right)^2=1.
$$

It follows that every profile curve, or meridian, $\{\theta=\theta_0\}$ with $\theta_0$ constant is an $h$-geodesic, and that a parallel $\{r=r_0\}$, with $r_0\in (0,2a)$ constant, is geodesic if and only if $m'(r_0) = 0$.
We observe that the geodesics equations implies
$$
\dfrac{d\theta(s)}{ds}m^2(r(s))=\nu,\ \text{where $\nu$ is constant},
$$
that is, the quantity $\frac{d\theta}{ds}m^2$ is conserved along the $h$-geodesics.
\begin{Lemma}(The Clairaut relation)
 Let $\gamma(s)=(r(s),\theta(s))$ be an $h$-unit speed geodesic on $(M,h)$. There exists a constant $\nu$ such that
$$
m(r(s))\cos\Phi(s)=\nu
$$
holds for any $s$, where $\Phi(s)$ denotes the angle between the tangent vector of $\gamma(s)$ and profile curve. 
\end{Lemma}
The constant $\nu$ is called the Clairaut constant of $\gamma$.

Several characterization of the cut locus of a Riemannian two-sphere of revolution are known (see \cite{BCST}, \cite{ST}, \cite{TASY}).

We recall the following important result from \cite{TASY}.

\begin{Proposition}\label{prop_1}
Let $h:[0,\pi]\to\R$ be a smooth function that can be extended to an odd smooth function on $\R$. If
\begin{enumerate}[(c1)]
\item $h(\pi-r)=\pi-h(r)$, for any $r\in[0,\pi]$;
\item $h'(r)>0$, for any $r\in\left[0,\frac{\pi}{2}\right)$;
\item $h''(r)>0$, for any $r\in\left(0,\frac{\pi}{2}\right)$,
\end{enumerate}
then
\begin{enumerate}[(i)]
\item the function $m:[0,\pi]\to\R$ given by $m(r):=a\sin h(r)$, where $a=\frac{1}{h'(0)}$, is the warp function of a two-sphere of revolution $M$.
\item Moreover, if $h''(r)>0$ on $\left(0,\frac{\pi}{2}\right)$, then the cut locus of a point $q=(r_0,0)\in M$ coincides with a subarc of the antipodal parallel $r=\pi-r_0$.
\end{enumerate}
\end{Proposition}

\begin{proof}
We give only the proof outline here, for details please consult \cite{TASY}. It can be seen that conditions (c1), (c2) imply that the function $m:[0,\pi]\to\R$ is positive, and $m(0)=0$, $m'(0)=1$, $m(\pi-r)=m(r)>0$ for $r\in(0,\pi)$, hence the two surface of revolution is well-defined.

Moreover, if (c3) holds good, then it can be proved that the half period function
$$
\varphi_m(\nu):=2\int_{m^{-1}(\nu)}^{\frac{\pi}{2}}\frac{\nu}{m(r)\sqrt{m^2(r)-\nu^2}}dr
$$
is decreasing, where $\nu$ is the Clairaut constant, hence the conclusion follows (see Lemma 1 and Proposition 1 in \cite{TASY}).

$\qedd$
\end{proof}

\begin{Remark}
Observe that $h(0)=0$, $h\left(\frac{\pi}{2}\right)=\frac{\pi}{2}$, $h(\pi)=\pi$, and the graph of $h$ looks like in Figure \ref{Fig2}.

\begin{figure}[H]
	\begin{center}
		\setlength{\unitlength}{0.2 cm}
		\begin{tikzpicture}[scale=1]
		\draw[->] (-1,0) -- (8,0) node[right] {$r$};
		\draw[->] (0,-1) -- (0,4) node[above] {$y$};
		\draw[smooth,color=blue,rotate=90,scale=1,xshift=1.25cm,yshift=-3.1cm] plot[domain=-.4*pi:.4*pi] (\x,{-tan(\x r)});

		\draw[-,dashed] (6.2,-1) -- (6.2,2.5) node[above] {$h(r)$};
		\draw[-,dashed] (6.2,2.5) -- (0,2.5) node[left] {$\pi$};
		\draw[-,dashed] (3.1,1.3) -- (0,1.3) node[left] {$\dfrac{\pi}{2}$};
		\draw[-,dashed] (3.1,-1) -- (3.1,1.3);
		\draw(0,0) node[below left]{$0$};
		\draw(6.2,-1.2) node[below]{$\pi$};
		\draw(3.1,-1.5) node{$\dfrac{\pi}{2}$};
		
		\end{tikzpicture}
	\end{center}
	\caption{The outline of the graph of $h$.}\label{Fig2}
\end{figure}
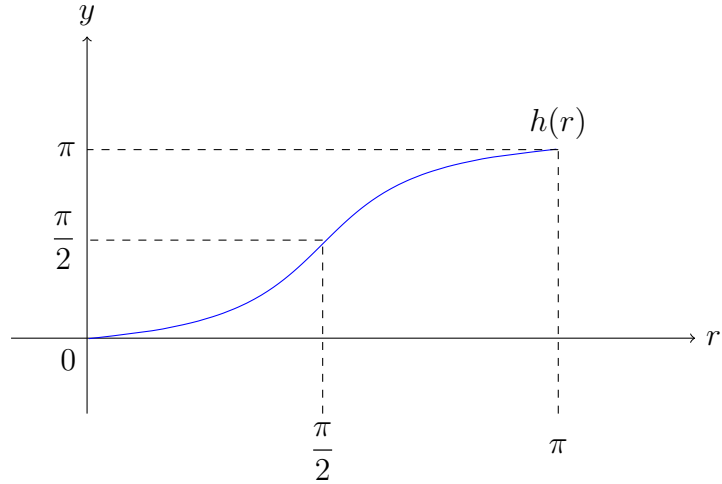
\end{Remark}

\begin{Definition}
A Riemannian (or Finsler) two-sphere of revolution whose cut locus is a subarc of a parallel will be called  with simple cut locus.
\end{Definition}

\begin{Remark}
This naming is related to graph theory in the sense that simple cut locus means that the cut locus is a simple graph with 2 vertices and one edge.
\end{Remark}
We recall some examples given in \cite{TASY}.

\begin{Example}\label{ex_1}
\begin{enumerate}[(i)]
\item If $h(r)=r-\alpha\sin(2r)$, for any $\alpha\in\left(0,\frac{1}{2}\right)$, one can see that 
$$
m(r)=a\sin(r-\alpha\sin(2r)).
$$
It follows that
\begin{equation*}
\begin{split}
m'(r)&=a\cos(r-\alpha\sin(2r))[1-2\alpha\cos(2r)],\\
m''(r)&=a\cos(r-\alpha\sin(2r))[-4\alpha\cos(2r)]-a[1-2\alpha\cos(2r)]\sin(r-\alpha\sin(2r))[1-2\alpha\cos(2r)].
\end{split}
\end{equation*}

Observe that the Gaussian curvature is
\begin{equation*}
\begin{split}
G(r)&=-\frac{m''(r)}{m(r)}\\&=\frac{1}{a\sin(r-\alpha\sin(2r))}
\left\{-a\cos(r-\alpha\sin(2r))[-4\alpha\cos(2r)]\right. \\
&\quad \left. +a[1-2\alpha\cos(2r)]\sin(r-\alpha\sin(2r))[1-2\alpha\cos(2r)]\right\}\\
&=4\alpha\cos(2r)\cot(r-\alpha\sin(2r))+[1-2\alpha\cos(2r)]^2,
\end{split}
\end{equation*}
which clearly is not monotone on $[0,\pi]$, see Figure \ref{fig_ex_1}.
On the other hand, it is easy to check that this $h$ satisfies conditions (c1), (c2), (c3) in Proposition \ref{prop_1}, hence it results that the Riemannian surface of revolution with the warp function $m$ has simple cut locus. 

\begin{figure}[H]
	\begin{center}
		\setlength{\unitlength}{0.2cm}
		\begin{tikzpicture}[scale=1]
		\draw[->] (0,-1) -- (0,4) node[above] {$y$};
		 \draw [->] (-2,0) -- (4,0) node[below]{$r$};
		 \draw [->] (0,-2) -- (0,4) ;
		\draw[smooth] plot[domain=0.07*pi:0.9*pi] (\x,{cos(2*\x r)*cot(\x r-0.25*sin(2*\x r))+(1-0.5*cos(2*\x r))^2});
\draw(0,0) node[below left]{$0$};

		\end{tikzpicture}
	\end{center}
	\caption{The graph of $G(r)$ in Example \ref{ex_1} (i), $r\in(0,\pi)$, where $\alpha=\frac{1}{4}$.}\label{fig_ex_1}
\end{figure}
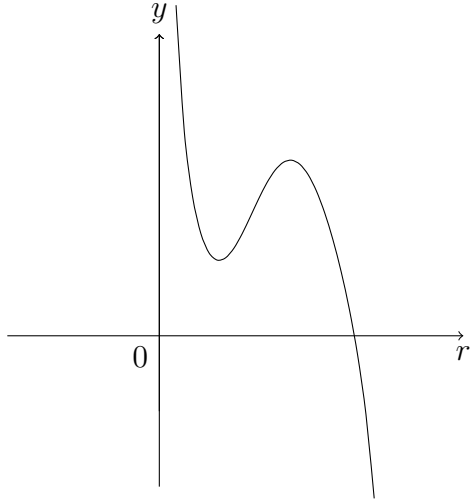 

\item If $h(r)=\arcsin\frac{\sin r}{\sqrt{1+\lambda\cos^2r}}$, for any $\lambda\geq 0$, it follows that
$$
m(r)=a\sin\left(\arcsin\frac{\sin r}{\sqrt{1+\lambda\cos^2r}}\right)=\frac{a\sin r}{\sqrt{1+\lambda\cos^2r}},
$$
therefore
\begin{equation*}
\begin{split}
m'(r)&=\frac{a}{1+\lambda\cos^2r}\left[\cos\sqrt{1+\lambda\cos^2r}
+\frac{\lambda\cos r\sin^2 r}{\sqrt{1+\lambda\cos^2 r}}
\right]
=\frac{a(1+\lambda)\cos r}{(1+\lambda\cos^2r)^{3/2}},\\
m''(r)&=\frac{a(1+\lambda)}{(1+\lambda\cos^2r)^3}\left[
-(1+\lambda\cos^2r)^{3/2}\sin r+3\lambda\cos^2 r\sin r (1+\lambda\cos^2r)^{1/2}
\right]\\
&=\frac{a(1+\lambda)}{(1+\lambda\cos^2r)^{5/2}}\left[
-\sin r(1+2\lambda\cos^2r+\lambda^2\cos^4r)+3\lambda\cos^2 r\sin r 
\right]\\
&=\frac{a(1+\lambda)\sin r}{(1+\lambda\cos^2r)^{5/2}}\left[
-1+\lambda\cos^2r-\lambda^2\cos^4r
\right].
\end{split}
\end{equation*}
We obtain the Gaussian curvature as follows
\begin{equation*}
\begin{split}
G(r)&=-\frac{m''(r)}{m(r)}
=\frac{(1+\lambda)(1-\lambda\cos^2r+\lambda^2\cos^4r)}{(1+\lambda\cos^2r)^2}
\end{split}
\end{equation*}
which again is not monotone on $[0,\pi]$, see Figure \ref{fig_ex_2}.

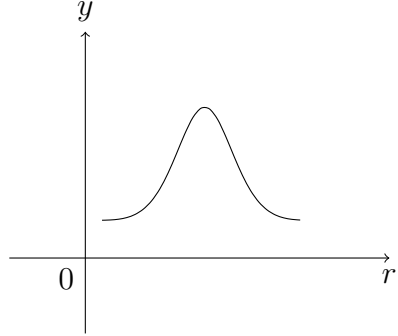
\begin{figure}[H]
	\begin{center}
		\setlength{\unitlength}{1cm}
		\begin{tikzpicture}[scale=1]
		 \draw [->] (-1,0) -- (4,0) node[below]{$r$};
		 \draw [->] (0,-1) -- (0,3) node[above] {$y$};
		\draw[smooth] plot[domain=0.07*pi:0.9*pi] (\x,{2*(1-cos(\x r)^2+cos(\x r)^4)/(1+cos(\x r)^2)^2});
\draw(0,0) node[below left]{$0$};

		\end{tikzpicture}
	\end{center}
	\caption{The graph of $G(r)$ in Example \ref{ex_1} (ii), $r\in(0,\pi)$, where $\lambda=1$.}\label{fig_ex_2}
\end{figure}
\end{enumerate}
This second example also  satisfies the conditions (c1), (c2), (c3)  in Proposition \ref{prop_1}. Hence it provides a two-sphere of revolution with simple cut locus.
\end{Example}

\begin{Remark}
A more complicated sequence of functions $h_n(r)$ with simple cut locus is constructed in \cite{TASY}, Theorem 1.
\end{Remark}

\section{Randers two-spheres of revolution}\label{sec: three Randers sphere}
We will show the existence of Randers two-spheres of revolution with simple cut locus using the following basic construction:

\begin{figure}[H]
\begin{center}
\setlength{\unitlength}{1cm}
\begin{picture}(12,1)

\put(0,0){$(M,h)$}
\put(0.5,-0.2){\vector(0,-1){2}}
\put(0,-2.7){$(M,F_0=\alpha_0+\beta_0)$}
\put(3.5,-2.6){\vector(1,0){5}}
\put(8.7,-2.7){$(M,F_1=\alpha_1+\beta_1)$}
\put(9.2,-2.2){\vector(0,1){2}}
\put(8.7,0){$(M,F_2=\alpha_2+\beta_2)$}

\put(-2,-1){$V_0,\|V_0\|_h<1$}
\put(-2,-1.7){$V_0$: $h$-Killing}

\put(4.3,-2.4){$V,F_0(-V)<1$}
\put(4.3,-3){$V$: $F_0$-Killing}

\put(9.4,-1){$W,F_1(-W)<1$}
\put(9.4,-1.7){$d\beta_2=0$}

\put(-2,-3.5){Navigation data: $(h,V_0)$}

\put(8,-3.5){Navigation data: $(h,V_0+V)$}

\put(8,0.7){Navigation data: $(h,V_0+V+W)$}

\end{picture}
\end{center}
\end{figure}
\bigskip
\vspace{3cm}
\noindent where $(M,h)$ is a Riemannian manifold, and $V_0,V,W\in\mathcal{X}(M)$ are vector fields on $M$.

It is known that in general, the navigation data $(h,V)$, where $(M,h)$ is a Riemannian metric and $V$ a vector field on $M$ such that $\|V\|_h<1$, induces the Randers metric
$$
F=\alpha(x,y)=\beta(x,y)=\frac{\sqrt{\lambda\|y\|^2_h+h(y,V)}}{\lambda}-\frac{h(y,V)}{\lambda}.
$$

Here $\lambda:=1-\|V\|^2_h$ and $h(y,V)=h_{ij}V^iy^j$ is the $h$-inner product of the vectors $V$ and $y$.

Conversely, the Randers metric $F=\alpha+\beta$, where $\alpha=\sqrt{a_{ij}(x)y^iy^j}$ is a Riemannian metric and $\beta=b_i(x)y^i$ a linear one-form on $TM$, induces the navigation data $(h,V)$ given by
$$
h^2=\varepsilon(\alpha^2-\beta^2),\ V=-\frac{1}{\varepsilon}\beta^\#.
$$

Here $h^2=h_{ij}(x)y^iy^j$, $\varepsilon:=1-\|b\|^2_\alpha$, and $\beta^\#$ is the Legendre transform of $\beta$, i.e.,
$$
\beta^\#=b_iy^i=a_{ij}b^iy^j
$$
(see \cite{BR}, \cite{BRS}, \cite{R} for details).

We recall some definitions for later use.

A vector field $X$ on $T^*M$ is called Hamiltonian vector field if there exists a smooth function $f:T^*M\to\R$, $(x,p)\mapsto f(x,p)$ such that
$$ X_f=\frac{\partial f}{\partial p_i}\frac{\partial}{\partial x^i}-\frac{\partial f}{\partial x^i}\frac{\partial}{\partial p_i}.$$

For instance, we can consider the Hamiltonian vector fields of the lift $W^*:=W^i(x)p_i$ of $W=W^i\frac{\partial}{\partial x^i}$ to $T^*M$, or of the Hamiltonian $\cK(x,p)$, the Legendre dual of any Finsler metric $F(x,y)$ on $M$ (see \cite{MHSS}).

Indeed, on a Finsler manifold $(M,F)$, for any $y\in T_xM\setminus\{0\}$ one can define
$$
p(y):=\frac{1}{2}\frac{d}{dt}\left[F^2(x,y+tv)\right]\big\vert_{t=0},\ v\in T_xM,
$$
and obtain in this way the map
\begin{equation*}
\begin{split}
\cL:TM &\to T^*M,\\
(x,y)&\mapsto (x,p),
\end{split}
\end{equation*}
called the Legendre transformation of $F$.

The curve $\hat{\gamma}(t)=(x(t),p(t)):[a,b]\to T^*M$ is called the integral curve (or sometimes the flow) of a Hamiltonian vector field $X_f\in \mathcal{X}(T^*M)$ if
$$
\frac{d\hat{\gamma}(t)}{dt}=X_{f}|_{\hat{\gamma}(t)}.
$$
More precisely, the mapping $\phi:\R\times T^*M\to T^*M$, $(t, (x,p))\mapsto 
\phi(t, (x,p))$, denoted also by $\phi_t(x,p)$ or $\phi_{(x,p)}t$, satisfying the properties
\begin{enumerate}[(i)]
\item $\phi(0,(x,p))=(x,p)$, for any $(x,p)\in T^*M$;
\item $\phi_s\circ \phi_t=\phi_{s+t}$, for all $s,t\in \R$;
\item $\dfrac{d\phi_{(x,p)}{t}}{dt}\vert_{t=0}=X\vert_{(x,p)}$,
\end{enumerate}
is called the one-parametric group, or simply the flow, of the vector field $X\in \mathcal{X}(T^*M)$. A given one-parametric group always induces a vector field $X\in \mathcal{X}(T^*M)$. Conversely, a given vector field $X\in \mathcal{X}(T^*M)$ induces only locally a one-parametric group, sometimes called the local flow of $X$. 

A smooth vector field $X\in \mathcal{X}(M)$ on a Finsler manifold $(M,F)$ is called F-Killing field if every local one-parameter transformation group $\{\varphi_t\}$ of $M$ generated by $X$ consists of local isometries of $F$. 
The vector field $X$ is F-Killing if and only if 
$L_{\widehat{X}}F=0$, where $L$ is the Lie derivative, and $\widehat{X}:=X^i\dfrac{\partial}{\partial x^i}+y^j\dfrac{\partial X^i}{\partial x^j}\dfrac{\partial}{\partial y^i}$ is the canonical lift of $X$ to $TM$, or, locally $X_{i|j}+X_{j|i}+2C_{ij}^pX_{p|q}y^q=0$, where `` $|$\,'' is the $h$-covariant derivative with respect to the Chern connection.

Moreover, in the Hamiltonian formalism,  the vector field $X$ on $M$ is Killing field with respect to $F$ if and only if 
	$$
	\{\mathcal{K},W^*\}=0,
	$$
	where $\mathcal{K}$ is the Legendre dual of $F$ (see \cite{MHSS}), $W^*=W^i(x)p_i$ and $\{\cdot,\cdot\}$ is the Poisson bracket.
\begin{Lemma}[Generalization of Hrimiuc-Shimada's result, see \cite{MHSS}]\label{lem_A1}
Let $(M,\widetilde{F}=\widetilde{\alpha}+\widetilde{\beta})$ be a Randers metric with general navigation data $(F=\alpha+\beta,W)$, $F(-W)<1$. Then the Legendre dual of $\widetilde{F}$ is $\widetilde{\cK}:T^*M\to\R$, $\widetilde{\cK}=\cK+W^*$, where $\cK$ is the Legendre dual of $F$ and $W^*=W^i(x)p_i$.
\end{Lemma}

\begin{proof}
Indeed, let $F=\alpha+\beta$ be a positive defined Randers metric on a differentiable manifold $M$ with indicatrix $\sum_F(x)=\{y\in T_xM:\ F(x,y)=1\}\subset T_xM $, and let $W\in\mathcal{X}(M)$ be a vector field such that $F(-W)<1$.

Let us denote by $\widetilde{\sum}(x):=\sum_F(x)+W(x)$ the rigid translation of $\Sigma_F(x)$ by $W(x)$, i.e.,
$$
\widetilde{\sum}(x):=\{y=u+W\in T_xM:\ F(u)=1\}.
$$

Firstly, observe that by rigid translation, the tangent vectors to $\sum_F$ and $\widetilde{\sum}$ remain parallel, i.e., there exists a smooth function $c(u)\neq 0$ such that
\begin{equation}\label{eq_proof_main_1}
\widetilde{Y}_{u+W_x}=c(u)(F_x)_{*,u},
\end{equation}
where $\widetilde{Y}_{u+W_x}$ is the tangent vector to $\widetilde{\sum}$ at $u+W_x$, and $(F_x)_{*,y}:T_y(T_xM)\to T\R\equiv\R$ is the tangent map of $F_x:T_xM\to[0,\infty)$, see Figure \ref{fig_rig_trans}.

\begin{figure}[H]
\begin{center}
\setlength{\unitlength}{1cm}
\begin{tikzpicture}
    \draw(0,0) circle(1.5);
    \draw(1,0.4) circle(1.5);
    \draw [->] (0,0) -- (1.5,0);
    \draw [->] (0,0) -- (0.5,0.5);
    \draw [dotted] (1.5, 0) -- (2.5,0.5);
    \draw [dotted] (0.5, 0.5) -- (2.5,0.5);
    \draw [->] (0,0) -- (2.5,0.5);
    \draw [->] (1.5,0) -- (1.5,1.5);
    \draw [->] (2.5,0.5) -- (2.5,2);
    \node at (-0.1,0.3) {$W$};
    \node at (0.7,-0.2) {$u$};
    \node at (3.2,0.5) {$u+W$};
    \node at (-2,-1) {$\Sigma_F(x)$};
    \node at (0,2) {$\widetilde{\Sigma}$};
    \node at (2.8,2.2) {$\widetilde{Y}$};
    \node at (1.05,1.5) {$(F_x)_*$};
	\node at (-2,2) {$T_xM$};
\end{tikzpicture}
\end{center}
\caption{The rigid translation of the indicatrix.}\label{fig_rig_trans}
\end{figure}
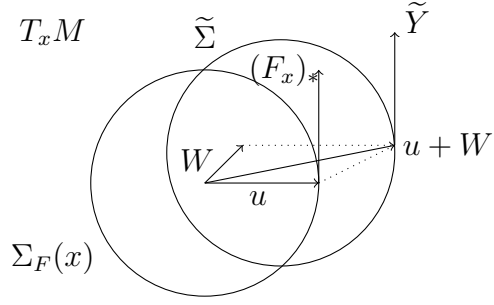

The solution of the Zermelo's navigation problem with data $(F,W)$ is a Finsler metric $\widetilde{F}$ such that
$$
\widetilde{F}_x(u+W_x)=1,
$$
where $u\in T_xM$, $F(x,u)=1$, and $\widetilde{F}_x$ is the restriction of $\widetilde{F}$ to $T_xM$. Since $\widetilde{\sum}$ is the rigid translation of $\sum$ such a Finsler metric must exist.

Second, with these notations, observe that in $T^*_xM$ we have
\begin{equation}\label{eq_proof_main_2}
\cL_{\widetilde{F}}(u+W_x)=c(u)\cL_F(u),
\end{equation}
where $\cL_{\widetilde{F}}$ and $\cL_F$ are the Legendre transformations of $\tilde{F}$ and $F$, respectively. This formula follows directly from \eqref{eq_proof_main_1} and the definition of the Legendre transformation.

Since relation \eqref{eq_proof_main_2} is between one-forms, actually this is a relation between linear transformations of the tangent space $T_xM$. If we pair \eqref{eq_proof_main_2} with $W_x$ and $u$, we get
\begin{equation}\label{eq_proof_main_2_1}
\langle\cL_{\tilde{F}(u+W_x)},W_x\rangle=c(u)\langle\cL_F(u),W\rangle
\end{equation}
and
\begin{equation}\label{eq_proof_main_2_2}
\langle\cL_{\tilde{F}(u+W_x)},u\rangle=c(u)\langle\cL_F(u),u\rangle=c(u),
\end{equation}
respectively, where we have used the fact that $F(u)=1$ is equivalent to $\langle\cL_F(u),u\rangle=1$. Here, $\langle\cdot,\cdot\rangle$ denotes the usual pairing of a one-form with a vector field.

Therefore, by the same reason, since $\tilde{F}(u+W)=1$ we have 
\begin{equation*}
\begin{split}
1&=\langle\cL_{\tilde{F}}(u+W_x),u+W_x\rangle=\langle\cL_{\tilde{F}}(u+W_x),u\rangle+\langle\cL_{\tilde{F}}(u+W_x),W_x\rangle=c(u)+c(u)\langle\cL_F(u),W\rangle,
\end{split}
\end{equation*}
where we use \eqref{eq_proof_main_2_1}, \eqref{eq_proof_main_2_2}. By the way, observe that
$$
c(u)=\frac{1}{1+\langle\cL_F(u),W\rangle}=\frac{1}{1+\langle u,W_x \rangle_{g_x(u)}},
$$
where $\langle\cdot,\cdot\rangle_{g_x(u)}$ is the inner product in $T_xM$ by $g_x(u)$, i.e. $\langle X,Y\rangle_{g_x(u)}=g_{ij}(x,u)X^iY^j$.

Next, let us denote by $\widetilde{\cK}$ and $\cK$ the Legendre dual metrics of $\widetilde{F}$ and $F$, respectively. It follows that
$$
1=\widetilde{\cK}[\cL_{\widetilde{F}}(u+W_x)]=c(u)\widetilde{\cK}(\cL_F(u)),
$$
and thus
\begin{equation*}
\begin{split}
\widetilde{\cK}(\cL_F(u))=\frac{1}{c(u)}=1+\langle\cL_F(u),W\rangle=\cK(\cL_F(u))+\langle\cL_F(u),W\rangle.
\end{split}
\end{equation*}

If we denote $\cL_F(u)=\omega_x=(x,p)\in T^*M$, then
\begin{equation}\label{eq_proof_main_2_2_1}
\widetilde{\cK}_x(p)=\cK_x(p)+\omega_x(W),
\end{equation}
where $\cK_x$ is the $\cL$-dual of $F=\alpha+\beta$.

Therefore, if $\widetilde{F}$ is {the solution of the Zermelo's navigation} (i.e. it is the rigid translation of the  indicatrix $\sum_F$ by $W$) with navigation data $(F,W)$, then
\begin{equation}\label{eq_proof_main_2_2_1*}
\widetilde{\cK}_x(p)=\cK_x(p)+W^*_x(p),
\end{equation}
where $\widetilde{\cK}$ and $\cK$ are the Hamiltonians of $\widetilde{F}$ and $F$, respectively, and $W^*=W^i(x)p_i$.

$\qedd$
\end{proof}

\begin{Lemma}\label{lem_A2}
Let $(M,F=\alpha+\beta)$ be a Randers metric, the vector field $W\in\mathcal{X}(M)$ with flow $\psi_t$. Then the Hamiltonian vector field $X_{\mathcal{K}}$ on $T^*M$ is invariant under the flow $\psi_{t,*}$ of $X_{W^*}$ if and only if $W$ is an $F$-Killing field, where $\mathcal{K}$ is the Legendre dual of $F$.
\end{Lemma}

\begin{proof}
Indeed, the invariance condition $\psi_{t,*}(X_{\cK})=X_{\cK}$ is equivalent to $\cL_{X_{W^*}}X_{\cK}=0$ by definition, hence $[X_{W^*},X_{\cK}]=0$, i.e. $X_{\{W^*,\cK\}}=0$. This shows that $W$ is actually $F$-Killing field.

$\qedd$
\end{proof}

\begin{Lemma}\label{lem_new_1}
Let $(M,F)$ be a Randers metric and $W\in TM$ a vector field on $M$. Then
\begin{enumerate}[(i)]
\item The navigation data of $\widetilde{F}$ is $(h,V+W)$, where $(h,V)$ is the navigation data of $F=\alpha+\beta$, and $\widetilde{F}$ is the solution of Zermelo's navigation problem for $(F,W)$.
\item The Randers metric $\widetilde{F}=\widetilde{\alpha}+\widetilde{\beta}$ is positive defined if and only if $F(-W)<1$.
\end{enumerate}
\end{Lemma}

\begin{proof}
\begin{enumerate}[(i)]
\item Recall that (see \cite{R}, \cite{BRS}) the indicatrix of $F$ is obtained by a rigid translation of the $h$-unit sphere $\sum_h(x)$ by $V$, i.e. for any $x\in M$
$$
\textstyle{\sum_F(x)=\sum_h(x)+V(x)},
$$
where $\sum_F(x)=\{y\in T_xM:\ F(x,y)=1\}$, $\sum_h(x)=\{y\in T_xM,\ \|y\|_h=1\}$, and $\|V\|_h<1$.
Then, if $\widetilde{F}$ is the solution of the Zermelo's navigation problem for $(F,W)$, we have
$$
\textstyle{\sum_{\widetilde{F}}(x)=\sum_F(x)+W(x)=\sum_h(x)+V(x)+W(x)},
$$
i.e., navigation data of $\widetilde{F}$ is $(h,V+W)$.

\item If we use (i), then $\widetilde{F}$ is positive defined Randers metric if and only if $\|V+W\|_h<1$.
Observe that
$$
\alpha^2(-W)=\alpha^2(W)=a_{ij}W^iW^j=\frac{1}{\lambda}h_{ij}W^iW^j+\left(\frac{V_i}{\lambda}W^i\right)^2=\frac{1}{\lambda}\|W\|_h^2+\frac{1}{\lambda^2}\langle V,W\rangle^2_h,
$$
where $\lambda=1-\|V\|^2_h>0$, and
$$
\beta(-W)=-\beta(W)=-b_iW^i=\frac{V_i}{\lambda}W^i=\frac{1}{\lambda}\langle V,W\rangle_h.
$$
It follows that 
$$
F(-W)=\sqrt{\frac{1}{\lambda}\|W\|_h^2+\frac{1}{\lambda^2}\langle V,W\rangle^2_h}+\frac{1}{\lambda}\langle V,W\rangle_h,
$$
hence $F(-W)<1$ is equivalent to
$$
\sqrt{\lambda\|W\|_h^2+\langle V,W\rangle_h^2}+\langle V,W\rangle_h<\lambda,
$$
where we use  $\lambda>0$ due to the fact that $F$ is positive defined Randers metric. Therefore, we successively obtain
\begin{equation*}
\begin{split}
\lambda\|W\|_h^2+\langle V,W\rangle_h^2&<\{\lambda-\langle V,W\rangle_h\}^2,\\
\lambda\|W\|_h^2+\langle V,W\rangle_h^2&<\lambda^2-2\lambda\langle V,W\rangle_h+\langle V,W\rangle_h^2,\\
\lambda\|W\|_h^2&<\lambda^2-2\lambda\langle V,W\rangle_h,\\
\|W\|_h^2&<\lambda-2\langle V,W\rangle_h,\\
\|W\|_h^2&<1-\|V\|_h^2-2\langle V,W\rangle_h,
\end{split}
\end{equation*}
which is equivalent to $\|V+W\|<1$, hence $\widetilde{F}$ is positive defined.
The converse implication is trivial.

\end{enumerate}

$\qedd$
\end{proof}

\begin{Lemma}\label{lem_new_2}
If $\widetilde{F}=\widetilde{\alpha}+\widetilde{\beta}$ is the Randers metric obtained in Lemma \ref{lem_new_1}, then we have
\begin{equation*}
\begin{split}
\widetilde{\alpha}^2&=\frac{1}{\eta}(\alpha^2-\beta^2)+\langle \frac{\widetilde{W}}{\eta},y\rangle_\alpha,\\
\widetilde{\beta}&=-\langle\frac{\widetilde{W}}{\eta},y\rangle,
\end{split}
\end{equation*}
where
\begin{equation*}
\begin{split}
\eta &:=[1+F(W)][1-F(-W)],\\
\widetilde{W}_i &:=W_i-b_i[1+\beta(W)],\ \textrm {and } W_i=a_{ij}W^j.
\end{split}
\end{equation*}
\end{Lemma}

\begin{proof}
Since the Zermelo's navigation data for $\widetilde{F}$ is $(h,U:=V+W)$, as shown in Lemma \ref{lem_new_1}, it follows (see \cite{R}, \cite{BR}) 
\begin{equation}\label{eq_proof_main_2_3_0}
\widetilde{a}_{ij}=\frac{1}{\sigma}h_{ij}+\frac{U_i}{\sigma}\frac{U_j}{\sigma},\quad 
\widetilde{b}_i=-\frac{U_i}{\sigma},
\end{equation}
where
\begin{equation*}
U_i=h_{ij}U^j=h_{ij}(V^i+W^i),\quad
\sigma :=1-\|V+W\|_h^2.
\end{equation*}

Recall that the navigation data $(h,V)$ of a Randers metric $F=\alpha+\beta$ can be computed by
\begin{equation*}
h_{ij}=\varepsilon(a_{ij}-b_ib_j),\quad
V^i=-\frac{b^i}{\varepsilon},
\end{equation*}
where $\varepsilon:=1-\|b\|_\alpha^2$, $b^i=a^{ij}b_j$ (see \cite{BR}, p. 233). Observe that as value $\varepsilon=1-\|b\|_\alpha^2=1-\|V\|_h^2=\lambda$.

We have
\begin{equation*}
\begin{split}
\langle V,W\rangle_h &= h_{ij}V^iW^j=\varepsilon(a_{ij}-b_ib_j)\left(-\frac{b^i}{\varepsilon}\right)W^j\\
&=-(a_{ij}b^iW^j-b_ib^ib_jW^j)-(b_jW^j-\|b\|_\alpha^2b_jW^j)\\
&=-(\beta(W)-\|b\|_\alpha^2\beta(W))=-\varepsilon\beta(W),
\end{split}
\end{equation*}
i.e.,
$$
\langle V,W\rangle_h=-\varepsilon\beta(W)
$$
and
$$
\|W\|_h^2=\varepsilon(a_{ij}-b_ib_j)W^iW^j=\varepsilon\{\alpha^2(W)-\beta^2(W)\}.
$$

It results
\begin{equation*}
\begin{split}
\sigma&=1-\|U\|_h^2=1-\|V\|_h^2-2\langle V,W\rangle_h-\|W\|_h^2\\
&=\varepsilon+2\varepsilon\beta(W)-\varepsilon\{\alpha^2(W)-\beta^2(W)\}\\
&=\varepsilon\{1+2\beta(W)+\beta^2(W)-\alpha^2(W)\}\\
&=\varepsilon\{[1+\beta(W)]^2-\alpha^2(W)\}\\
&=\varepsilon[1+\beta(W)+\alpha(W)][1+\beta(W)-\alpha(W)]\\
&=\varepsilon[1+F(W)][1-F(-W)],
\end{split}
\end{equation*}
i.e.,\newpage
\begin{equation}\label{eq_new_3.4}
\sigma=\varepsilon\eta,
\end{equation}
where $\eta=[1+F(W)][1-F(-W)]$.

Moreover, we have
\begin{equation*}
\begin{split}
U_i&=h_{ij}U^j=h_{ij}(V^j+W^j)=\varepsilon(a_{ij}-b_ib_j)V^j+\varepsilon(a_{ij}-b_ib_j)W^j\\
&=\varepsilon(a_{ij}-b_ib_j)\left(-\frac{b^j}{\varepsilon}\right)+\varepsilon(a_{ij}-b_ib_j)W^j\\
&=-[b_i-b_i\|b\|_\alpha^2]+\varepsilon[W_i-b_i\beta(W)]\\
&=-\varepsilon b_i+\varepsilon[W_i-b_i\beta(W)]\\
&=\varepsilon\{W_i-b_i[1+\beta(W)]\}=\varepsilon\widetilde{W}_i,
\end{split}
\end{equation*}
i.e., $U=\varepsilon \widetilde{W}$.

With these results, we compute
\begin{equation*}
\begin{split}
\widetilde{a}_{ij}&=\frac{1}{\sigma}h_{ij}+\frac{U_i}{\sigma}\frac{U_j}{\sigma}\\
&=\frac{1}{\varepsilon\eta}\varepsilon(a_{ij}-b_ib_j)+\frac{\varepsilon\widetilde{W}_i}{\varepsilon\eta}\frac{\varepsilon\widetilde{W}_j}{\varepsilon\eta}\\
&=\frac{1}{\eta}(a_{ij}-b_ib_j)+\frac{\widetilde{W}_i}{\eta}\frac{\widetilde{W}_j}{\eta}
\end{split}
\end{equation*}
and
$$
\widetilde{b}_i=-\frac{U_i}{\sigma}=-\frac{\varepsilon\widetilde{W}_i}{\varepsilon\eta}=-\frac{\widetilde{W}_i}{\eta},
$$
hence the conclusion follows.

$\qedd$
\end{proof}

\begin{Remark}
We observe that $\widetilde{F}=\widetilde{\alpha}+\widetilde{\beta}$ is positive defined if and only if $\|\widetilde{b}\|_{\widetilde{\alpha}}<1$, i.e., $\sigma=1-\|U\|_h^2=1-\|\widetilde{b}\|_{\widetilde{\alpha}}^2>0$.

On the other hand, \eqref{eq_new_3.4} implies that
$$
\sigma>0\ \Leftrightarrow\ \varepsilon[1+F(W)][1-F(-W)]>0\ \Leftrightarrow\ 1-F(-W)>0,
$$
since $\varepsilon>0$ due to the fact that $F$ is assumed positive defined and $F(W)>0$.

In other words, we have shown that
$$
F(-W)<1\ \Leftrightarrow\ \|\widetilde{b}\|_{\widetilde{\alpha}}<1,
$$
that is another proof and more a intuitive explanation of positive definiteness condition $F(-W)<1$ (compare to \cite{FM}). 
\end{Remark}

We will show a generic result on geodesics, conjugate and cut loci of a Randers metric.

\begin{Lemma}\label{lem: X}
Let $(M,F=\alpha+\beta)$ be a not flat Randers metric, let  $W\in\mathcal{X}(M)$ be a vector field on $M$ such that $F(-W)<1$ and let $\widetilde{F}=\widetilde{\alpha}+\widetilde{\beta}$ be the solution of navigation problem for $(F,W)$. If $W$ is $F$-Killing field, then
\begin{itemize}
\item[(i)] the $\widetilde{F}$-unit speed geodesics $\widetilde{\cP}$ are given by
$$
\widetilde{\cP}(t)=\psi_t(\cP(t)),
$$
where $\cP$ is an $F$-unit speed geodesic and $\psi_t$ is the flow of $W$;
\item[(ii)] the point $\widetilde{\cP}(l)$ is conjugate to $q=\widetilde{\cP}(0)$ along the $\widetilde{F}$-geodesic $\widetilde{\cP}:[0,l]\to M$ if and only if the point $\cP(l)$ is conjugate to $q=\cP(0)$ along the corresponding $F$-geodesic $\cP(t)=\psi_{-t}(\widetilde{\cP}(t))$, for $t\in[0,l]$;
\item[(iii)] the point $\hat{p}$ is an $\widetilde{F}$-cut point of $q$ if and only if $p=\psi_{-l}(\hat{p})$ is an $F$-cut point of $q$, 
\end{itemize}
where $l=d_{\widetilde{F}}(q,\hat{p})$.
\end{Lemma}
\begin{proof}
We will prove (i).

For simplicity, if we also denote by $\psi_t:T^*M\to T^*M$ the flow of $X_{W^*}$, then for a curve $\cP(t)$ on $T^*M$ we denote
$$
\hat{\cP}(t)=\psi_t(\cP(t)),
$$
i.e., we map $\cP(t)\mapsto \hat{\cP}(t)$ by the flow $\psi_t$.

By taking the tangent map
$$
(\psi_{t,*})_{\cP(t)}:T_{\cP(t)}(T^*M)\to T_{\hat{\cP}(t)}(T^*M),
$$
we have
$$
X\big|_{\cP(t)}\mapsto (\psi_{t,*})_{\cP(t)}(X\big|_{\cP(t)})=(\psi_{t,*}X)_{\hat{\cP}(t)},
$$
for any vector field $X$ on $T^*M$.

If $\cP(t)$ is an integral curve of the Hamiltonian vector field $X_\cK$, i.e. $\frac{d\cP(t)}{dt}=X_\cK\big|_{\cP(t)}$, where $\cK$ is the Legendre dual of $F$, then the derivative formula of a function of two variables give
\begin{equation*}
\begin{split}
\frac{d}{dt}(\hat{\cP}(t))&=\frac{d}{dt}\psi(t,\cP(t))=X_{W^*}\big|_{\cP(t)}+\psi_{t,*}\left(\frac{d\cP(t)}{dt}\right)\\
&=X_{W^*}\big|_{\hat{\cP}(t)}+\psi_{t,*}\left(X_{\cK}\big|_{\cP(t)}\right)\\
&=X_{W^*}\big|_{\hat{\cP}(t)}+\left(\psi_{t,*}X_\cK\right)_{\hat{\cP}(t)}\\
&=X_{W^*}\big|_{\hat{\cP}(t)}+\left(X_\cK\right)_{\hat{\cP}(t)}\\
&=\left(X_{W^*+\cK}\right)_{\hat{\cP}(t)}=\left(X_{\widetilde{\cK}}\right)_{\hat{\cP}(t)},
\end{split}
\end{equation*}
where we have used that the Legendre dual of 
$\widetilde{F}$ is $\widetilde{\cK}=\cK+W^*$, and $\psi_{t,*}X_\cK=X_\cK$ (see Lemmas \ref{lem_A1} and \ref{lem_A2}),  hence (i) is proved. 

Next, we will prove (ii).

If we denote by $\cP_s:[0,l]\to M$, $-\varepsilon<s<\varepsilon$ a geodesic variation of the $F$-geodesic $\cP$, such that all curves in the variation are $F$-geodesics, then we obtain the variation vector field
$$
J:=\frac{\partial \cP_s}{\partial s}\big|_{s=0},
$$
which clearly is an $F$-Jacobi field.

Taking now into account (i), which shows that
$$
\widetilde{J}=\psi_*(J)
$$
is a Jacobi vector field along $\widetilde{\cP}$, hence the conjugate points along $\cP$ and $\widetilde{\cP}$ correspond each other under the flow $\psi_t$ of $W$, hence (ii) is proved.

Finally, we will prove (iii). From (ii) it is easy to see that since $W$ is $F$-Killing field, the arclength parameter of the $F$-geodesic $\cP$ and of the $\widetilde{F}$-geodesic $\widetilde{\cP}$ coincide.

It can be seen, like in the Riemannian case, that the points where the distance function $d_F(p,\cdot)$ looses its differentiability coinciding by the flow $\psi_t$ to the points where the distance function $d_{\widetilde{F}}(p,\cdot)$ looses its differentiability (see \cite{ST2016}, Theorem A for the characterization of cut points in terms of differentiability of distance function). Hence, (iii) follows.

$\qedd$
\end{proof}

\begin{Lemma}\label{lem: X2}
Let $(M,F=\alpha+\beta)$ be a Randers metric with navigation data $(h,W)$. The followings are equivalent 
\begin{enumerate}[(i)]
\item $d\beta=0$,
\item $dW^\#=d\log\lambda\wedge W^\#$, 
\end{enumerate}
where the one-form $W^\#$ is the $h$-Legendre transformation of $W$ and $\lambda=1-\|W\|^2_h$.
\end{Lemma}
\begin{proof}
Indeed, observe that from the Zermelo's navigation formulas we get (see for instance \eqref{eq_proof_main_2_3_0}, or \cite{R}, \cite{BR}, \cite{BRS}) we get
$$
\beta=-\frac{W_i}{\lambda}dx^i=-\frac{1}{\lambda}W^\#,
$$
where $W^\#=\cL_hW$. Here, $\cL_h$ is the Legendre transform with respect to $h$.

By differentiation, we get
$$
d\beta=-d\left(\frac{1}{\lambda}W^\#\right)=-\left[-\frac{1}{\lambda^2}d\lambda\wedge W^\#+\frac{1}{\lambda}dW^\#\right]=-\frac{1}{\lambda}\left[-d\log\lambda\wedge W^\#+dW^\#\right],
$$
hence the desired equivalence follows. 

$\qedd$
\end{proof}

Summing up, here is our main result.

\begin{theorem}\label{thm_main_1}

Let $(M,h)$ be a Riemannian manifold and let $V_0,V,W\in\mathcal{X}(M)$ be vector fields on $M$. If $\|V_0\|_h<1$, we denote by $F_0=\alpha_0+\beta_0$ the  positive defined Randers metric obtained as solution of the  Zermelo's navigation problem $(h,V_0)$.
\begin{enumerate}[(i)]
\item
	\begin{enumerate}
	\item[(i.1)]
	If $F_0(-V)<1$, then $F_1=\alpha_1+\beta_1$ is a positive defined Randers metric, where $F_1$ is the solution of Zermelo's navigation problem $(F_0,V)$.
	\item[(i.2)] 
	If $F_1(-W)<1$, then $F_2=\alpha_2+\beta_2$ is a positive defined Randers metric, where $F_2$ is the solution of Zermelo's navigation problem $(F_1,W)$.
	\end{enumerate}
\item
	\begin{enumerate}	
	\item[(ii.1)]
	The Randers metric $F_1=\alpha_1+\beta_1$ is the solution of Zermelo's navigation problem  $(h,V_0+V)$.
	\item[(ii.2)]
The Randers metric $F_2=\alpha_2+\beta_2$ is the solution of Zermelo's navigation problem  $(h,V_0+V+W)$.
	\end{enumerate}
\item If the following conditions are satisfied
	\begin{enumerate}
	\item[(C0)]
	$V_0$ is $h$-Killing,
	\item[(C1)]
	$V$ is $F$-Killing,
	\item[(C2)]
	$d(V_0+V+W)^\#=d\log\widetilde{\lambda}\wedge(V_0+V+W)$,
	\end{enumerate}
where $(V_0+V+W)^\#=\mathcal{L}_h(V_0+V+W)$ is the Legendre transformation of $V_0+V+W$ with respect to $h$, and $\widetilde{\lambda}:=1-\|V_0+V+W\|^2_h$,
	then
	\begin{enumerate}
	\item[(iii.1)]
	The $F_0$-unit speed geodesics $\cP_0$, and the $F_1$-unit speed geodesics $\cP_1$ are given by
	\begin{equation}\label{eq_thm_main_1}
	\begin{split}
	\cP_0(t)&=\varphi_t(\rho(t)),\\
	\cP_1(t)&=\psi_t(\cP_0(t))=\psi_t\circ\varphi_t(\rho(t)),
	\end{split}
	\end{equation}
	where $\rho(t)$ is an $h$-unit speed geodesic and $\varphi_t$ and $\psi_t$ are the flows of $V_0$, and $V$, respectively.
	
	The $F_2$-unit speed geodesic $\cP_2(t)$ coincides as points set with $\cP_1(t)$.
	\item[(iii.2)]
	The conjugate points of $q=\cP_2(0)$ along the $F_2$-geodesic $\cP_2$ coincide to the conjugate points of $q=\cP_1(0)$ along the $F_1$-geodesic $\cP_1$, up to parameterization.
	The point $\cP_1(l)$ is conjugate to $q=\cP_1(0)$ along the $F_1$-geodesic $\cP_1:[0,l]\to M$ if and only if the point $\cP_0(l)$ is conjugate to $q=\cP_0(0)$ along the corresponding $F_0$-geodesic $\cP_0(t)=\psi_{-t}(\cP_1(t))$, for $t\in[0,l]$.
	The point $\cP_0(l)$ is conjugate to $q=\cP_0(0)$ along the $F_0$-geodesic $\cP_0:[0,l]\to M$ if and only if the point $\rho(l)$ is conjugate to $q=\rho(0)$ along the corresponding $h$-geodesic $\rho(t)=\varphi_{-t}(\cP_0(t))$, for $t\in[0,l]$, where $\varphi_t$, and $\psi_t$ are the flows of $V_0$, and $V$, respectively.
	\item[(iii.3)]
	The $F_2$-cut locus of $q$ coincide as points set with the $F_1$cut locus of $q$, up to parameterization.
	
	The point $\hat{p}_1$ is an $F_1$-cut point of $q$, if and only if $\hat{p}_0=\psi_{-l}(\hat{p}_1)$ is an $F_1$-cut point of $q$, where $l=d_{F_1}(q,\hat{p}_1)$. The point $\hat{p}_0$ is an $F_0$-cut point of $q$, if and only if $p_0=\varphi_{-l}(\hat{p}_0)$ is an $h$-cut point of $q$, where $l=d_{F_0}(q,\hat{p}_0)$.
	\end{enumerate}

\end{enumerate}

\end{theorem}

\begin{proof}[Proof of (i), (ii)]
The proof of (i.1), (ii.1) follows immediately from Lemma \ref{lem_new_1} for $(F_0,V)$. Likewise, (i.2) and (ii.2) follows from Lemma \ref{lem_new_1} for $(F_1,W)$.

\noindent {\it Proof of (iii).}
 The proof will be given in two steps. 
 
\textbf{Step 1.} (Properties of $F_0$, $F_1$) With the notations in hypothesis, conditions (C0), (C1) imply that the geodesics, conjugate points and cut points of the Randers metrics  $F_0$, $F_1$ have the properties in (iii) due to Lemma \ref{lem: X}.
 
\textbf{Step 2.} (Properties of  $F_2$)
 By taking into account Lemma \ref{lem: X2} one can see that condition (C2) is actually equivalent to $d\widetilde{\beta}=0$, that is the Randers metrics $F_1=\alpha+\beta$ and $F_2=\widetilde{\alpha}+\widetilde{\beta}$ are projectively related (\cite{BCS}), therefore having the same geodesics as non-parameterized curves, same conjugate points and same cut points. Hence, the desired properties of $F_2$ follows (see Figure \ref{fig_unit_speed}).

\begin{figure}[H]
\begin{center}
\setlength{\unitlength}{1cm}
\begin{tikzpicture}[scale=1]
	\coordinate (O) at (0,0,0);
	\coordinate (A) at (6,2,0);
	\coordinate (B) at (6,-2,0);
	\coordinate (C) at (6,-3,0);
		\draw [->,xshift=1cm,yshift=0.95cm] (0,0) -- (0.1,0.1);
		\draw [->,xshift=1cm,yshift=-0.95cm] (0,0) -- (0.1,-0.1);
		\draw [->,xshift=1cm,yshift=-1.245cm] (0,0) -- (0.1,-0.1);
		\draw[] (O) to [bend left=30] (A);
		\draw[] (O) to [bend right=30] (B);
		\draw[] (O) to [bend right=30] (C);
		\draw[smooth,rotate=-90,dashed,xshift=0.5cm,yshift=2.5cm] plot[domain=-0.85*pi:0.75*pi] (\x,{-sin(\x r)+2});
			\node at (-0.5,0) {$q$};
			
			\draw[->]  (3.8,-2.85) -- (3.3,-2.4);
			\draw[->]  (3.5,-2.1) -- (4,-1.7);
			\draw[->]  (5,2.1) -- (4.5,3);
			
			\draw [->,xshift=5cm,yshift=0.05cm] (0,0) -- (0.1,0.1);
			\draw [->,xshift=3.6cm,yshift=-2.5cm] (0,0) -- (-0.05,0.1);
			
			\node at (7,2.5) {$\{\cP_2(t)\}=\{\cP_1(t)\}$};
			\node at (9,2) {$\cP_1(t)=\psi_t(\cP_0(t))=\psi_t\circ\varphi_t(\rho(t))$};
			\node at (7.5,-2) {$\cP_0(t)=\varphi_t(\rho(t))$};
			\node at (4.5,3.3) {$W$};
			\node at (6.5,-3) {$\rho(t)$};
			\node at (5.5,0) {$\psi_t$};
			\node at (4,-1.5) {$V$};
			\node at (4,-2.5) {$\varphi_t$};
			\node at (3.5,-3) {$V_0$};
			\node at (0,0) {\textbullet};
			\node at (5,2.1) {\textbullet};
			\node at (3.5,-2.1) {\textbullet};
			\node at (3.8,-2.85) {\textbullet};
\end{tikzpicture}
\end{center}
\caption{The unit-speed geodesics $\rho,\rho_0,\rho_1$ and $P_2$ of $h,F_0,F_1$ and $F_2$, respectively.}\label{fig_unit_speed}
\end{figure}
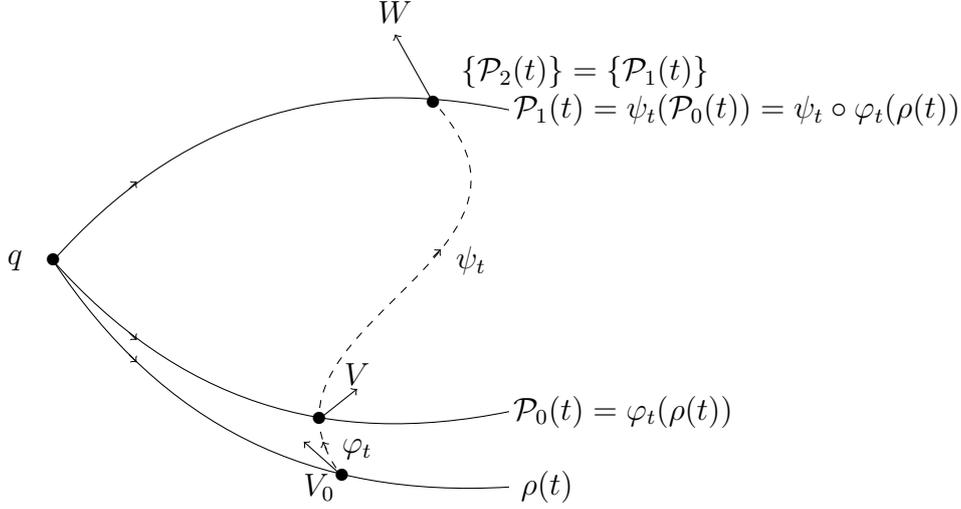

$\qedd$
\end{proof}

\begin{Remark}\label{rem_condition}
Under the hypotheses in the Theorem \ref{thm_main_1} we have
\begin{enumerate}[(i)]
\item conditions (C0), (C1) are equivalent to (C1), (C2)$'$, where 

(C2)$'$: $V$ is $h$-Killing;
\item if we replace conditions (C0), (C1), (C2) with 

(C3): $V_0+V+W$ is $h$-Killing,
\end{enumerate}
then the $F_2$-geodesics, conjugate locus and cut locus is obtained from the $h$-geodesics, conjugate locus and cut locus deformed by the flow of $V_0+V+W$, respectively. Observe that in this case, the $F_2$-geodesics, conjugate locus and cut locus are different from these in Theorem \ref{thm_main_1}.

\end{Remark}

\begin{Remark}\label{rem: sequence}
The construction presented here can now be extended to a sequence of Finsler metrics.

Our sequence construction has two steps. Let $(M,h)$ be a Riemannian two-sphere of revolution, and $V_0,V_1,\ldots,V_{k-1},W_k,\ldots,W_n\in\mathcal{X}(M)$, $n=k-l$, a sequence of vector fields on $M$.

\textbf{Step 1.} A sequence of vector fields: $V_0,V_1,\ldots,V_{k-1}$, such that all $V_i$ are $F_i$-Killing fields, $i\in\{0,1,\ldots,k-1\}$,
\begin{figure}[H]
\begin{center}
\setlength{\unitlength}{1cm}
\begin{picture}(20,2)
{\small
\put(1,1.1){\vector(1,0){2.2}}
\put(6.1,1.1){\vector(1,0){2.5}}
\put(11.6,1.1){\vector(1,0){2.5}}

\put(0,1){$(M,h)$}
\put(3.3,1){$(M,F_1=\alpha_1+\beta_2)$}
\put(8.7,1){$(M,F_2=\alpha_2+\beta_2)$}

\put(1.1,1.4){$V_0,\|V_0\|_h<1$}
\put(1.1,0.6){$h$-Killing}
\put(6.2,1.4){$V_1,F_1(-V_1)<1$}
\put(6.2,0.6){$F_1$-Killing}
\put(11.7,1.4){$V_2,F_2(-V_2)<1$}
\put(11.7,0.6){$F_2$-Killing}

\put(14.2,1.1){$\ldots$}

\put(0,-0.5){$\ldots$}
\put(0.5,-0.5){\vector(1,0){3.8}}
\put(0.6,-0.2){$V_{k-1},F_{k-1}(-V_{k-1})<1$}
\put(0.6,-0.9){$F_{k-1}$-Killing}
\put(4.3,-0.6){$(M,F_k=\alpha_k+\beta_k)$;}

}
\end{picture}
\end{center}
\end{figure}
\bigskip
\textbf{Step 2.} A sequence of vector fields: $W_k.\ldots,W_l$, such that each $\beta_j$ is closed one-form, for $j\in\{k,k+1,\ldots,l\}$.

\begin{figure}[H]
\begin{center}
\setlength{\unitlength}{1cm}
\begin{picture}(20,2)
{\small
\put(2.5,1.1){\vector(1,0){2.6}}
\put(8.6,1.1){\vector(1,0){3.5}}

\put(0,1){$(M,F_k+\alpha_k+\beta_k)$}
\put(5.1,1){$(M,F_{k+1}=\alpha_{k+1}+\beta_{k+1})$}
\put(5.6,0.6){$d\beta_{k+1}=0$}
\put(12.1,1){$(M,F_{k+2}=\alpha_{k+2}+\beta_{k+2})$}
\put(12.6,0.6){$d\beta_{k+2}=0$}

\put(2.6,1.3){$W_k,F_k(-W_k)<1$}
\put(8.7,1.3){$W_{k+1},F_{k+1}(-W_{k+1})<1$}

\put(0,-0.5){$\ldots$}
\put(0.5,-0.5){\vector(1,0){3.5}}
\put(0.6,-0.3){$W_{k+2},F_{k+2}(-W_{k+2})<1$}
\put(4.1,-0.5){$\ldots$}
\put(4.6,-0.5){\vector(1,0){3.7}}
\put(4.7,-0.3){$W_{n-1},F_{n-1}(-W_{n-1})<1$}
\put(8.3,-0.6){$(M,F_n=\alpha_n+\beta_n)$.}
\put(8.8,-1.1){$d\beta_{n}=0$}

}
\end{picture}
\end{center}
\end{figure}

\end{Remark}

\bigskip

Theorem \ref{thm_main_1} can be naturally extended to the two-step construction above. Indeed, if we start with a Riemannian structure $(M,h)$ and a sequence of vector fields $V_0,V_1,\ldots,V_{k-1}\in\mathcal{X}(M)$, the Zermelo's navigation problems for
\begin{equation*}
\begin{split}
(h,V_0) &\ \text{with solution}\ F_1=\alpha_1+\beta_1,\\
(F_1,V_1) &\ \text{with solution}\ F_2=\alpha_2+\beta_2,\\
&\quad\quad\vdots\\
(F_{k-1},V_{k-1}) &\ \text{with solution}\ F_k=\alpha_k+\beta_k,
\end{split}
\end{equation*}
will generate a sequence of positive defined Randers metrics provided $\|V_0\|_h<1$, $F_i(-V_i)<1$, $i\in\{1,\ldots,k-1\}$. The Zermelo's navigation data for $F_i$ is also $(h,V_0+\ldots+V_i)$, for all $i\in\{1,\ldots,k-1\}$, hence $F_k$ is positive defined if and only if $\|V_0+\ldots+V_{k-1}\|_h<1$.

Next, if we start with $(M,F_k)$ and the sequence of vector fields $W_k,\ldots,W_{n-1}\in \mathcal{X}(M)$ the Zermelo's navigation problems for
\begin{equation*}
\begin{split}
(F_k=\alpha_k+\beta_k,W_k) &\ \text{with solution}\ F_{k+1}=\alpha_{k+1}+\beta_{k+1},\\
(F_{k+1}=\alpha_{k+1}+\beta_{k+1},W_{k+1}) &\ \text{with solution}\ F_{k+2}=\alpha_{k+2}+\beta_{k+2},\\
&\quad\quad\vdots\\
(F_{n-1},W_{n-1}) &\ \text{with solution}\ F_n=\alpha_n+\beta_n,
\end{split}
\end{equation*}
will generate another sequence of positive defined Randers metrics provided $F_{k+j}(-W_{k+j})<1$, $j\in\{0,1,2,\ldots,n-k-1\}$.

Observe again that by combining these with the sequence of Randers metrics constructed at first step, we can easily see that the Zermelo's navigation data of $F_{k+j}$, $j\in\{0,1,\ldots,n-k\}$ is $(h,V_0+\ldots+V_{k-1}+W_k+\ldots+W_{k+j})$, hence the final Randers metric $F_n=\alpha_n+\beta_n$ is positive defined if and only if 
$$
\left\|\sum_{i=0}^{k-1}V_i+\sum_{j=0}^{n-k-1}W_{j+k}\right\|_h<1.
$$

Moreover, if we impose conditions
\begin{itemize}
\item[($C_0$)] $V_0$ is $h$-Killing;
\item[($C_{1i}$)] $V_i$ is $F_i$-Killing, $i\in\{1,\ldots,k-1\}$;
\item[($C_{2j}$)] $W_{k+j}$ is chosen such that $d\beta_{k+j}=0$, $j\in\{0,\ldots,n-k\}$. 
\end{itemize}

Clearly the geodesics, conjugate and cut loci of $F_n$ can be obtained from the geodesics, conjugate locus, cut locus of $h$ through the flow of $V:=\sum_{i=0}^{k-1}V_i$, respectively.

Observe that condition $(C_{2j})$ are similar to (C2) in Theorem \ref{thm_main_1}, but we prefer not to write them here explicitly, for simplicity.

This is the generalization of Theorem \ref{thm_main_1} to the sequence of Finsler metrics $\{F_1,\ldots,F_n\}$.

Nevertheless, there is a shortcut in this construction in the spirit of Remark \ref{rem_condition}. Indeed, if $V+W$ is $h$-Killing, where $V=\sum_{i=0}^{k-1}V_i$, $W=\sum_{j=0}^{n-k-1}W_{k+j}$, then the geodesics, conjugate and cut loci of $F_n$ are obtained from the geodesics, conjugate  and cut loci of $h$ through the flow of $V+W$, respectively.

\section{Conclusions}\label{sec: four Examples}
We will consider a simple example of the construction described in Theorem \ref{thm_main_1}.

\begin{figure}[H]
\begin{center}
\setlength{\unitlength}{1cm}
\begin{picture}(20,3)

\put(1.2,1.7){\vector(1,0){1.8}}
\put(6.3,1.7){\vector(1,0){1.8}}
\put(11.4,1.7){\vector(1,0){1.8}}

\put(0,1.6){$(M,h)$}
\put(3,1.6){$(M,F_0=\alpha_0+\beta_0)$}
\put(8.1,1.6){$(M,F_1=\alpha_1+\beta_1)$}
\put(13.2,1.6){$(M,F_2=\alpha_2+\beta_2)$}

\put(3,0.5){$(h,V_0)$}
\put(8.1,0.5){$(h,V_0+V)$}
\put(13.2,0.5){$(h,V_0+V+W)$}

\put(1.2,2){$V_0=\mu_o\frac{\partial}{\partial\theta}$}
\put(1.2,1){$h$-Killing}
\put(6.3,2){$V=\mu_1\frac{\partial}{\partial\theta}$}
\put(6.3,1){$h$-Killing}
\put(11.7,2){$W$}
\put(13.7,1){$d\beta_2=0$.}
\end{picture}
\end{center}
\end{figure}

Let us start with the Riemannian two-sphere of revolution $(M\simeq \Sph^2,h=dr^2+m^2(r)d\theta^2)$ given in Section \ref{sec_two_sphere}, Proposition \ref{prop_1}. The vector field $V_0\in\mathcal{X}(M)$ is $h$-Killing if and only if it is a rotation, i.e. $V_0=\mu_0\frac{\partial}{\partial\theta}$, $\mu_0$ constant, where $(r,\theta)$ are the $h$-geodesic coordinates. In order that $F_0$ is positive defined we need the condition $\|V_0\|_h<1$, i.e. $\mu_0^2m^2(r)<1$.

Next, we consider the vector field $V\in\mathcal{X}(M)$ which is also $h$-Killing if and only if $V=\mu_1\frac{\partial}{\partial\theta}$ with $\mu_1$ constant (see Remark \ref{rem_condition}). The Randers metric $F_1=\alpha_1+\beta_1$ is positive defined if and only if $(\mu_0+\mu_1)^2m^2(r)<1$, i.e. we choose $\mu_0,\mu_1$ such that $m(r)<\frac{1}{\mu_0+\mu_1}$.

Finally, we construct a vector field $W\in\mathcal{X}(M)$ such that $d\beta=0$. For instance
$$
W=A(r)\frac{\partial}{\partial r}-\mu\frac{\partial}{\partial\theta},
$$
where $\mu:=\mu_0+\mu_1$ is an obvious choice. Observe that $V_0+V+W=A(r)\frac{\partial}{\partial r}$, hence $\beta_2=-\frac{A(r)}{1-A^2(r)}dr$. If we impose condition $A^2(r)<1$, then $F_2$ is a positive defined Randers metric.

We obtain
\begin{Proposition}\label{prop: examples}
Let $(M\simeq \Sph^2,h=dr^2+m^2(r)d\theta^2)$ be the Riemannian two-sphere of revolution described in Proposition \ref{prop_1}. Let 
$$
V_0=\mu_0\frac{\partial}{\partial\theta},\ V=\mu_1\frac{\partial}{\partial\theta},\ W=A(r)\frac{\partial}{\partial r}-\mu\frac{\partial}{\partial\theta},
$$
be three vector fields on $M$, where $\mu=\mu_0+\mu_1$.
\begin{enumerate}[(i)]
\item If $m(r)<\frac{1}{\mu}$ for all $r\in [0,\pi]$ and $A:[0,\pi]\to[0,\infty)$ is smooth function such that $A^2(r)<1$, then the Finsler metrics $F_0=\alpha_0+\beta_0$, $F_1=\alpha_1+\beta_1$, $F_2=\alpha_2+\beta_2$, obtained as solutions of Zermelo's navigation problem with data $(h,V_0)$, $(F_0,V)$ and $(F_1,W)$, respectively, are positive defined Randers metrics.
\item The Randers metrics $F_1=\alpha_1+\beta_1$ and $F_2=\alpha_2+\beta_2$ can be obtained as solutions of Zermelo's navigation problem with data $\left(h,\mu\frac{\partial}{\partial \theta}\right)$ and $(h,A(r)\frac{\partial}{\partial r})$, respectively.
\item 
\begin{itemize}
\item[iii.1] The unit speed $F_2$-geodesics are given by
$$
\mathcal{P}(t)=\psi_t(\rho(t)),
$$ 
where $\rho$ are unit speed $h$-geodesics and $\psi_t$ is the flow of $\widetilde{V}=V_0+V+W=A(r)\frac{\partial}{\partial r}$.
\item[iii.2] The point $\widehat{p}=\mathcal{P}(l)$ is conjugate to $\widehat{q}:=\mathcal{P}(0)$ along the $F_2$-geodesic $\cP:[0,l]\to M$ if and only if $q=\mathcal{P}(0)=\rho(0)$ is conjugate to $p:=\rho(l)$ along the corresponding $h$-geodesic $\rho(t)=\psi_{-t}(\cP(t))$, $t\in[0,l]$.
\item[iii.3] The cut locus of a point $q\in(M,F_2)$ is a subarc if the antipodal parallel displaced by the flow $\varphi_t$.
\end{itemize}
\end{enumerate}
\end{Proposition}

One can describe the Finsler metric $F_2$ in coordinates as follows. If $h$ is given by $ds^2=dr^2+m^2(r)d\theta^2$ in the geodesic coordinates $(r,\theta)\in(0,\pi]\times[0,2\pi)$, then
\begin{equation*}
\begin{split}
\alpha^2_2&=\frac{1}{\widetilde{\lambda}^2(r)}dr^2+\frac{m(r)}{\widetilde{\lambda}}ds^2,\\
\beta_2&=-\frac{A(r)}{\widetilde{\lambda}(r)}dr,\ \widetilde{\lambda}(r):=1-A^2(r).
\end{split}
\end{equation*}

More precisely, if $m(r)=\frac{1}{1-2\alpha}\sin(r-\alpha\sin 2r)$ (see Example \ref{ex_1}) , for any $\alpha\in\left(0,\frac{1}{2}\right)$ and $A(r):=\frac{r}{\sqrt{r^2+1}}$, then $\widetilde{\lambda}=\frac{1}{r^2+1}$ hence the Finsler metric $F_2$ is given by
\begin{equation*}
\begin{split}
\alpha_2^2&=(r^2+1)^2dr^2+\frac{1}{1-2\alpha}(r-\alpha\sin 2r)(r^2+1)d\theta^2,\\
\beta_2&=-r\sqrt{r^2+1}\ dr,\ r\in(0,\pi],\ \theta\in[0,2\pi).
\end{split}
\end{equation*}
Other examples can be similarly constructed from the Riemannian examples in \cite{TASY}.

\begin{Remark}
Observe that in order to construct Randers metric having same cut locus structure as the Riemannian metric $h$, another condition is also possible. Indeed, choosing
\begin{equation*}
\begin{split}
V_0:&=v_0(r,\theta)\frac{\partial}{\partial r}+w_0(r,\theta)\frac{\partial}{\partial \theta},\\
V:&=-v_0(r,\theta)\frac{\partial}{\partial r}+[\mu-w_0(r,\theta)]\frac{\partial}{\partial\theta},
\end{split}
\end{equation*}
will lead to $V_0+V=\mu\frac{\partial}{\partial \theta}$ which is $h$-Killing and combined with $W=A(r)\frac{\partial}{\partial r}-\mu\frac{\partial}{\partial\theta}$ the derived conclusion follows, for any smooth functions $v_0,w_0$ and constant $\mu$ such that $m(r)<\frac{1}{\mu}$.
\end{Remark}

\end{document}